\documentclass[11pt]{article}

\usepackage{amssymb,amsmath,amsfonts,amsthm}
\usepackage{latexsym}
\usepackage{graphics}
\usepackage{indentfirst}

\newtheorem*{thm*}{Theorem}
\newtheorem{thm}{Theorem}
\newtheorem{lem}[thm]{Lemma}
\newtheorem{pro}[thm]{Proposition}

\newtheorem{cor}[thm]{Corollary}

\newtheorem{ques}[thm]{Question}

\newcommand{\N}{\mathbb{N}}

\begin{document}

\title{A Simple Characterization of Proportionally 2-choosable Graphs}

\author{Hemanshu Kaul\footnotemark[1], Jeffrey A. Mudrock\footnotemark[1], Michael J. Pelsmajer\footnotemark[1], and Benjamin Reiniger\footnotemark[1]}

\footnotetext[1]{Department of Applied Mathematics, Illinois Institute of Technology, Chicago, IL 60616.  E-mail:  {\tt {kaul@iit.edu, jmudrock@hawk.iit.edu, pelsmajer@iit.edu, breiniger@iit.edu.}}}

\date{2018}

\maketitle

\begin{abstract}

We recently introduced proportional choosability, a new list analogue of equitable coloring. Like equitable coloring, and unlike list equitable coloring (a.k.a. equitable choosability), proportional choosability bounds sizes of color classes both from above and from below.  In this note, we show that a graph is proportionally 2-choosable if and only if it is a linear forest such that its largest component has at most~5 vertices and all of its other components have two or fewer vertices.  We also construct examples that show that characterizing equitably 2-choosable graphs is still open.

\medskip

\noindent {\bf Keywords.} graph coloring, equitable coloring, list coloring, equitable choosability.

\noindent \textbf{Mathematics Subject Classification.} 05C15.

\end{abstract}

\section{Introduction}\label{intro}

All graphs in this note are assumed to be finite, simple graphs unless otherwise noted.  Generally speaking we follow West~\cite{W01} for basic terminology and notation.

\subsection{Equitable Coloring and List Coloring}

The study of equitable coloring began with a conjecture of Erd\H{o}s~\cite{E64} in~1964, but the general concept was formally introduced by Meyer~\cite{M73} in~1973.  An \emph{equitable $k$-coloring} of a graph $G$ is a proper $k$-coloring of $G$, $f$, such that the sizes of the color classes differ by at most one (where a $k$-coloring has exactly $k$, possibly empty, color classes).  It is easy to see that for an equitable $k$-coloring, the color classes associated with the coloring are each of size $\lceil |V(G)|/k \rceil$ or $\lfloor |V(G)|/k \rfloor$.  We say that a graph $G$ is \emph{equitably $k$-colorable} if there exists an equitable $k$-coloring of $G$. Equitable colorings are useful when it is preferable to form a proper coloring without under-using any colors or using any color excessively often. Equitable coloring has found many applications, see for example~\cite{T73}, \cite{P01}, \cite{KJ06}, and~\cite{JR02}.

List coloring was introduced independently by Vizing~\cite{V76} and Erd\H{o}s, Rubin, and Taylor~\cite{ET79} in the 1970's.  A list assignment $L$ for a graph $G$ is a function which associates with each vertex of $G$ a list of colors.  When $|L(v)|=k$ for all $v \in V(G)$ we say that $L$ is a \emph{k-assignment} for $G$.  We say $G$ is \emph{$L$-colorable} if there exists a proper coloring $f$ of $G$ such that $f(v) \in L(v)$ for each $v \in V(G)$ (and $f$ is called a \emph{proper $L$-coloring} of $G$).  We say $G$ is \emph{$k$-choosable} if a proper $L$-coloring of $G$ exists whenever $L$ is a $k$-assignment for $G$.
List coloring and equitable coloring have each been widely studied in the decades since they were introduced.

In 2003, Kostochka, West, and the third author introduced a list analogue of equitable coloring~\cite{KP03}, {\it equitable choosability}, which has received some attention as well (see~\cite{KK13,LB09,ZW11,ZB08,ZB10,ZB15}).  Suppose that $L$ is a $k$-assignment for the graph $G$.  A proper $L$-coloring of $G$ is \emph{equitable} if each color appears on at most $\lceil |V(G)|/k \rceil$ vertices.  Such a coloring is called an \emph{equitable $L$-coloring} of $G$, and we call $G$ \emph{equitably $L$-colorable} when an equitable $L$-coloring of $G$ exists.  We say $G$ is \emph{equitably $k$-choosable} if $G$ is equitably $L$-colorable whenever $L$ is a $k$-assignment for $G$.  Note that $\lceil |V(G)|/k \rceil$ is the same upper bound on the size of color classes in typical equitable coloring.  So, for the notion of equitable choosability, no color is used excessively often.

We recently introduced a new list analogue of equitable coloring, called proportional choosability, which also bounds the sizes of the color classes from below.  Note that there is no lower bound on the number of lists that contain a color: a color might appear in exactly one list, for example.  Thus, one cannot hope to have a lower bound like $\lfloor |V(G)|/k \rfloor$ for color class sizes.  Instead, we used bounds that vary with the availability of each color, as follows.  (See~\cite{KM18} for further discussion.)

\subsection{Proportional Choosability}

Suppose that $L$ is a $k$-assignment for graph $G$.  The \emph{palette of colors associated with $L$} is $\cup_{v \in V(G)} L(v)$.  For each color $c$, the \emph{multiplicity of $c$ in $L$} is the number of lists of $L$ in which $c$ appears.  The multiplicity of $c$ in $L$ is denoted by
$\eta_L(c)$ (or simply $\eta(c)$ when the list assignment is clear), so $\eta_L(c)=\left\lvert{\{v : v \in V(G), c \in L(v) \}}\right\rvert$.  Throughout this note, when $L$ is a list assignment for some graph $G$, we always use $\mathcal{L}$ to denote the palette of colors associated with $L$.

Given a $k$-assignment $L$ for a graph $G$, a proper $L$-coloring, $f$, of $G$ is a \emph{proportional $L$-coloring} of $G$ if for each $c \in \mathcal{L}$, $f^{-1}(c)$, the color class of $c$, is of size
$ \left \lfloor {\eta(c)/k} \right \rfloor$ or $\left \lceil {\eta(c)/k} \right \rceil$.
We say that graph $G$ is \emph{proportionally $L$-colorable} if a proportional $L$-coloring of $G$ exists, and $G$ is \emph{proportionally $k$-choosable} if $G$ is proportionally $L$-colorable whenever $L$ is a $k$-assignment for $G$.

In the rest of this section, we state some important basic properties of proportional $k$-choosability.  Proofs and further discussion of these properties can be found in~\cite{KM18} and~\cite{M18}.

For the first property, note that if a graph is $k$-choosable, then it is $k$-colorable.  On the other hand, it can happen that a graph is equitably $k$-choosable, but not equitably $k$-colorable. 

\begin{pro}[\cite{KM18}] \label{pro: motivation}
If $G$ is proportionally $k$-choosable, then $G$ is both equitably $k$-choosable and equitably $k$-colorable.
\end{pro}

The next two properties state that proportional $k$-choosability is {\it monotone in $k$} and it is also a {\it monotone graph property}.
Note that the analogous statements for $k$-colorability and $k$-choosability are easily verified, while the analogous statements for equitable $k$-colorability and equitable $k$-choosability are not true.

\begin{pro}[\cite{KM18}] \label{pro: monoink}
If $G$ is proportionally $k$-choosable, then $G$ is proportionally $(k+1)$-choosable.
\end{pro}

\begin{pro}[\cite{KM18}] \label{lem: monotone}
Suppose $H$ is a subgraph of $G$.  If $G$ is proportionally $k$-choosable, then $H$ is proportionally $k$-choosable.
\end{pro}

For any graph $G$, let $\chi_{pc}(G)$ denote smallest $k$ such that $G$ is proportionally $k$-choosable; this is called the \emph{proportional choice number} of $G$.  
By Proposition~\ref{pro: monoink}, a graph $G$ is proportionally $k$-choosable if and only if $k \geq \chi_{pc}(G)$.

Lastly, we present a simple upper bound.

\begin{pro}[\cite{KM18}] \label{pro: orderplus1}
For any graph $G$ other than the complete graph, $\chi_{pc}(G) \leq |V(G)|-1$.
\end{pro}

\subsection{Results and Questions}
We now summarize our results and mention some open questions.   In Section~\ref{easyresults} we will present some basic results related to proportional choosability which will be used in the proof of our main result.

In Section~\ref{characterize} we prove our main result; that is, we completely characterize graphs that are proportionally 2-choosable.  Such a characterization is natural to seek since there are well known simple characterizations of 2-colorable graphs, equitably 2-colorable graphs, and 2-choosable graphs (see~\cite{ET79}).

Interestingly, finding a nice characterization of equitably 2-choosable graphs is open.  In 2004 Wang and Lih claimed that a connected graph $G$ is equitably 2-choosable if and only if (1) $G$ is 2-choosable and (2) $G$ has a bipartition $X,Y$ such that $||X|-|Y|| \leq 1$ (see~\cite{LW04}).  While the forward direction of their claim is true, the other direction of their claim is not even true for trees, as the following counterexample illustrates.

\vspace{5mm}

\noindent \textbf{Counterexample:} Suppose we form tree $T$ by taking a path on three vertices, $u,v,w$, and add $k$ pendant edges to $u$, $k$ pendant edges to $w$, and $2k-1$ pendant edges to $v$ for some $k \geq 3$.  Then, $T$ is a 2-choosable graph (see the characterization of 2-choosable graphs in~\cite{ET79}) that has a bipartition with partite sets of equal size.  However, if $L$ is the 2-assignment for $T$ that assigns $\{1,2\}$ to $v$ and leaves adjacent to $v$, $\{1,3\}$ to $u$ and $\{2,3\}$ to leaves adjacent to $u$, and $\{2,3\}$ to $w$ and $\{1,3\}$ to leaves adjacent to $w$, then $T$ is not equitably $L$-colorable which implies $T$ is not equitably 2-choosable.

\vspace{5mm}

We prove the following (perhaps surprisingly restrictive) characterization of proportionally 2-choosable graphs in Section~\ref{characterize}.  A \emph{linear forest} is a forest where all the components are paths.

\begin{thm} \label{thm: full2character}
Graph $G$ is proportionally 2-choosable if and only if $G$ is a linear forest such that the largest component of $G$ has at most 5 vertices and all the other components of $G$ have 2 or fewer vertices.
\end{thm}

Notice that Theorem~\ref{thm: full2character} tells us that $\chi_{pc}(P_n) > 2$ whenever $n \geq 6$.  As it turns out, we do not understand the proportional choosability of paths very well.  Specifically, we have the following open questions.  Question~\ref{ques: proportionalpaths2} was asked by Kostochka~\cite{K18} during a talk given by the second author~\cite{M182}.

\begin{ques} \label{ques: proportionalpaths1}
For each $n \geq 6$, what is the value of $\chi_{pc}(P_n)$?
\end{ques}

\begin{ques} \label{ques: proportionalpaths2}
Does there exist a constant $C$ such that $\chi_{pc}(P_n) \leq C$ for every $n \in \N$?
\end{ques}

For proportional choosability it is not enough to consider connected graphs because there are examples where the disjoint union of proportionally $k$-choosable graphs is not itself proportionally $k$-choosable (see Proposition~\ref{pro: maxdegreesec} below).  However, the particular graph structure evident in Theorem~\ref{thm: full2character} leads us to the following question.

\begin{ques} \label{ques: tackiton}
Suppose $G$ is proportionally $k$-choosable.  If $H$ is a graph that is vertex disjoint from $G$ with $|V(H)| \leq k$, must it be the case that the disjoint union of these graphs, $G+H$, is proportionally $k$-choosable?
\end{ques}

By Theorem~\ref{thm: full2character}, we know that the answer to Question~\ref{ques: tackiton} is yes for $k=2$.

\section{Basic Results} \label{easyresults}

In this section we present some results that will be used in Section~\ref{characterize}.  The proofs of the results in this section are presented in both~\cite{KM18} and~\cite{M18}.

Suppose that $L$ is a $k$-assignment for the graph $G$, and suppose that $f$ is a proportional $L$-coloring of $G$.  We call $a \in \mathcal{L}$ a \emph{well distributed color in $L$} (or simply a \emph{well distributed color} when the list assignment is clear) if $\eta(a)$ is divisible by $k$.  A color $p \in \mathcal{L}$ is called \emph{perfectly used with respect to $f$} if $p$ is well distributed and $|f^{-1}(p)|= \eta(p)/k$.  A color $b \in \mathcal{L}$ is called \emph{almost excessive with respect to $f$} if $b$ is not well distributed and $|f^{-1}(b)| = \lceil \eta(b)/k \rceil$, and a color $d \in \mathcal{L}$ is called \emph{almost deficient with respect to $f$} if $d$ is not well distributed and $|f^{-1}(d)| = \lfloor \eta(d)/k \rfloor$.

\begin{lem}[\cite{KM18}] \label{lem: countexcessive}
Suppose that $L$ is a $k$-assignment for the graph $G$, and suppose that $f$ is a proportional $L$-coloring of $G$.
Then the number of almost excessive colors with respect to $f$ is
$$ \frac{1}{k} \sum_{l \in \mathcal{L}} r_l,$$
where $r_l$ is $\eta(l) \bmod k$.
\end{lem}

Propositions~\ref{lem: monotone} and~\ref{pro: starsec} immediately yield Corollary~\ref{cor: maxdegreesec}.

\begin{pro}[\cite{KM18}] \label{pro: starsec}
$K_{1,2k-1}$ is not proportionally $k$-choosable for each $k \in \N$.
\end{pro}

\begin{cor}[\cite{KM18}] \label{cor: maxdegreesec}
Suppose that $G$ is a graph with $\Delta(G) \geq 2k-1$ for some $k \in \N$.  Then, $G$ is not proportionally $k$-choosable.  That is,
$$\chi_{pc}(G) > \frac{\Delta(G)+1}{2}$$
for any graph $G$.
\end{cor}

By Corollary~\ref{cor: maxdegreesec}, proportionally 2-choosable graphs have maximum degree at most~2.  This fact will be used in Section~\ref{characterize}.  

We will also need the following result on the proportional choosability of the disjoint union of stars.

\begin{pro}[\cite{KM18}] \label{pro: maxdegreesec}
Suppose $H_1, H_2, \ldots, H_k$ are $k$ pairwise vertex disjoint copies of $K_{1,k}$.  If $G = \sum_{i=1}^k H_i$, then $G$ is not proportionally $k$-choosable.
\end{pro}

Since Proposition~\ref{pro: orderplus1} implies $K_{1,k}$ is proportionally $k$-choosable, Proposition~\ref{pro: maxdegreesec} shows that adding more components to a graph may make finding a proportional coloring more difficult.  Thus, techniques we use to prove results about proportional choosability of connected graphs may not work in the context of disconnected graphs.  

Finally, we will use the following result to observe that cycles of order~4 are not proportionally 2-choosable.

\begin{pro}[\cite{KM18}] \label{pro: completebipartitesec}
$K_{m,m}$ is not proportionally $m$-choosable for each $m \in \N$.
\end{pro}

\section{Characterizing Proportionally 2-Choosable Graphs} \label{characterize}

In this section we will prove Theorem~\ref{thm: full2character}.  Suppose that $G$ is proportionally 2-choosable.  Corollary~\ref{cor: maxdegreesec} implies that $\Delta(G) \leq 2$.  Since $G$ must be 2-colorable, this implies that $G$ is the disjoint union of paths and even cycles.  Proposition~\ref{pro: completebipartitesec} implies that $G$ does not contain a $C_4$.  This means $G$ is the disjoint union of paths and even cycles of order at least 6.  Finally, Proposition~\ref{pro: maxdegreesec} implies that $G$ does not contain a $K_{1,2} + K_{1,2}$.  Since a copy of $K_{1,2} + K_{1,2}$ is a subgraph of every path of order at least 6 and every cycle of order at least 6, Proposition~\ref{lem: monotone} implies that $G$ is a linear forest with largest component of order at most 5 and all other components of order at most 2.

We spend the remainder of this section proving the other direction of Theorem~\ref{thm: full2character}.

For any $2$-assignment $L$ for a graph $G$, we call a color $a \in \mathcal{L}$ \emph{odd} if $\eta(a)$ is odd and \emph{even} if $\eta(a)$ is even.  A proportional $L$-coloring of $G$ requires even colors to be perfectly used while each odd color $c$ must be used $(\eta(c)-1)/2$ or $(\eta(c)+1)/2$ times.  It is easy to see that the number of odd colors in $\mathcal{L}$ is even, and Lemma~\ref{lem: countexcessive} implies that exactly half of the odd colors in $\mathcal{L}$ are used $(\eta(c)+1)/2$ times in a proportional $L$-coloring.

For the sake of contradiction, let $G$ be a minimal counterexample. Then $G$ is not proportionally 2-choosable, and every proper subgraph of $G$ is proportionally 2-choosable.  Proposition~\ref{pro: orderplus1} implies $|V(G)| \geq 4$.  Furthermore, $G$ has one component, $P$, that is a copy of $P_k$ with $k\le 5$, and its other components are copies of $P_1$ or $P_2$.  There is at least one $2$-assignment $L$ for $G$ for which there is no proportional $L$-coloring of $G$.

\begin{lem} \label{lem: disjoint}
(i) If $u,v$ are adjacent vertices, then $L(u)$ and $L(v)$ are not disjoint.
\\
(ii) If $\eta(c)=2$ for some $c\in {\mathcal L}$, then the vertices whose lists contain $c$ are not adjacent.
\end{lem}

\begin{proof}
If $L(u)$ and $L(v)$ are disjoint, remove the edge $uv$ and apply induction.
The proportional $L$-coloring obtained still works when you add $uv$ back to the graph.

If $\eta(c)=2$ and $c\in L(u)\cap L(v)$ and $u$ is adjacent to $v$, remove the edge $uv$ and apply induction.
Since $\eta(c)=2$, exactly one of $u,v$ will get color $c$, so the proportional $L$-coloring obtained still
works when you add $uv$ back to the graph.
\end{proof}

\begin{lem} \label{lem: doubleodd}
Suppose that $v\in V(G)$ and $L(v)=\{c_1,c_2\}$ such that both $c_1,c_2$ are odd with respect to $L$.
Then $v$ must have distinct neighbors $x,y$ with $c_1\in L(x)$ and $c_2\in L(y)$.
\end{lem}

\begin{proof}
Let $G' = G - \{v \}$ and $L'$ be the 2-assignment for $G'$ by restricting the domain of $L$ to $V(G')$.  By choice of $G$, there is a
proportional $L'$-coloring $f'$ of $G'$.  Since $c_1,c_2$ are odd with respect
to $L$ and both appear in $L(v)$, they are even with respect to $L'$.
Thus, $c_1$ and $c_2$ are used on
exactly $(\eta_L(c_1)-1)/2$ and $(\eta_L(c_2)-1)/2$ vertices by $f'$, respectively.

If $v$ does not have distinct neighbors $x,y$ with $c_1\in L(x)$ and $c_2\in L(y)$,
then both colors do not appear on neighbors of $v$ in $f'$.
Without loss of generality, assume $c_1$ is not used on any neighbor of $v$ by $f'$.
Then color $v$ by $c_1$; this extends $f'$ to a proper $L$-coloring $f$ of $G$.
Since $\eta_L$ and $\eta_{L'}$ are identical except for $c_1$ and $c_2$ and $f$
uses $c_1$ on $(\eta(c_1)+1)/2$ vertices and $c_2$ on $(\eta(c_2)-1)/2$ vertices of $G$,
$f$ is a proportional $L$-coloring of $G$.
\end{proof}

For any subgraph $H$ of $G$, when we say ``a list of $H$'' we mean $L(v)$ for some $v\in V(H)$.  When we say ``the lists of $H$'' we mean every list in the set $\{L(v) : v \in V(H) \}$.
\begin{lem} \label{lem: camouflage}
We can assume that for each odd color $c$ with respect to $L$
and for each component $H$ of $G$,
$c$ is not in exactly one of the lists of $H$.
\end{lem}

\begin{proof}
Over all 2-assignments for $G$ for which there is no proportional coloring, choose $L$ to be one which minimizes the number of such colors.
Assume that there is at least one such color $c_1$: an odd color for which there is a component $H$ of $G$ such that $c_1\in L(v)$ for exactly one vertex $v\in V(H)$.

Let $c_2$ be the other color in $L(v)$.  By Lemma~\ref{lem: doubleodd}, $c_2$ is not odd.
The number of odd colors in $\mathcal L$ is even.  So, there is another color $c_3\in \mathcal{L}$ which is odd.
Let $L'$ be identical to $L$ except that $L'(v)=\{c_2,c_3\}$ instead of $\{c_1,c_2\}$.  Then $c_1$ and $c_3$ are both even with respect to $L'$,
so by the choice of $L$, there must exist a proportional $L'$-coloring $f'$ of $G$.  Both $c_1$ and $c_3$ are perfectly used by $f'$, so $c_1$ is used $(\eta_L(c_1)-1)/2$
times and $c_3$ is used $(\eta_L(c_3)+1)/2$ times by $f'$.  If $f'(v)=c_2$ then $f'$ is a proportional $L$-coloring of $G$.  If
$f'(v)=c_3$ then recoloring $v$ with $c_1$ yields a proportional $L$-coloring of $G$.  This contradicts the choice of $L$.
\end{proof}

Next, suppose that $v_1,v_2$ are the vertices of a $P_2$ component or of two $P_1$ components of $G$.
If $L(v_1)=L(v_2)$, remove those vertices (from $G$ and from $L$) and apply induction, then use
one of the colors on $v_1$ and the other on $v_2$.

If $L(v_1)=\{a,b\}$ and
$L(v_2)=\{b,c\}$, replace those vertices by a new vertex $v_*$ with $L(v_*) =
\{a,c\}$ and apply induction.  Then take the color assigned to $v_*$ and give it to
$v_1$ or $v_2$, and put $b$ on the other vertex.  In both situations we obtained
a proportional $L$-coloring of $G$---a contradiction---so we can assume that $L(v_1)$ and $L(v_2)$ are disjoint.
By Lemma~\ref{lem: disjoint}(i), $G[\{v_1,v_2 \}]$ cannot be a $P_2$ component.  So, $G$ has no $P_2$ components.

Let ${\mathcal L}_1$ be the union of all lists of $P_1$ components.
Each color in ${\mathcal L}_1$ is even by Lemma~\ref{lem: camouflage} and
it is in the list of only one $P_1$ component; so it must be in an odd number of the
lists of $P$.  Furthermore $P$ cannot be a $P_1$ component (nor a $P_2$ component), so $3\le k\le 5$.

Suppose that there is a color $a \in {\mathcal L}_1$ which is in exactly one list of
$P$.  Let $u$ and $v$ be the vertices of the $P_1$ and $P_k$ components, respectively, whose
lists contain $a$.  If $L(u)=L(v)$, remove both $u$ and $v$ and apply induction, then
color $v$ with $a$ and color $u$ with $b$ to get a proportional $L$-coloring of $G$.
Otherwise, we may suppose that
that  $L(u)=\{a,b\}$ and $L(v)=\{a,c\}$.  Remove $u$ and replace $a$ in $L(v)$
by $b$ so that $L(v)=\{b,c\}$, then apply induction.  If $v$ is colored $b$,
then change it to $a$ and color $u$ with $b$.  If $v$ is colored $c$, then
color $u$ with $a$.  Either way, we get a proportional $L$-coloring of $G$.
Thus we can assume that every color in ${\mathcal L}_1$ appears in 3 or 5 lists of $P$.

Any odd color in $\mathcal L$ is in more than one list of $P$ by Lemma~\ref{lem: camouflage}.
Thus, if a color is in an odd number of lists of $P$
(either because it is an odd color or it is in ${\mathcal L}_1$),
it must appear in 3 or 5 lists of $P$.  Since $|V(P)| \leq 5$, there are at most two such colors.
Since the number of odd colors is even and $|\mathcal{L}_1|$ is even, there are either zero or two such colors.
We consider the zero case first.

Suppose every color in $\mathcal{L}$ appears in an even number of lists of $P$.  This implies that every color in $\mathcal{L}$ is even and ${\mathcal L}_1$ is empty.  So, $G=P$.
Two consecutive vertices must share a color $a$ in their lists by Lemma~\ref{lem: disjoint}(i), but $\eta(a)\not=2$ by Lemma~\ref{lem: disjoint}(ii), so $\eta(a)=4$.
If $k=4$ then either the other color of the lists alternate between two colors $b,c$ by Lemma~\ref{lem: disjoint}(ii),
or every list is the same; both cases are easily proportionally $L$-colored.
If $k=5$ then there is one vertex $v$ without $a$ in its list.
By Lemma~\ref{lem: disjoint}, $L(v)$ shares a color $b$ with a neighbor's list and $\eta(b)\not=2$, so $\eta(b)=4$.
There is one more color $c$ which has $\eta(c)=2$; let $u,w$ be the vertices with lists that contain $c$.
Skipping $u$, we can alternate colors $a$ and $b$ along the rest of $P$, since $L(w)$ contains $a$ or $b$.
Then $a$ and $b$ are used twice each; coloring $u$ by $c$ completes a proportional $L$-coloring.

In the remaining cases, there are exactly two colors $a,b$ that appear in an odd number of lists of $P$.
These colors appear in $3$ and $3$, or $3$ and $5$, or $5$ and $5$ lists of $P$.

If $a$ and $b$ each appear in 5 lists of $P$, then
$k=5$, there may be one $P_1$ component or none, and
$|{\mathcal L}|=2$. In both cases, an ordinary 2-coloring is a proportional $L$-coloring.

In the next case, suppose that $a$ and $b$ appear in 5 and 3 lists of $P$, respectively.
Thus, $k=5$, there may be one isolated vertex $v$ with $L(v)=\{a,b\}$ (or
none), and there is a color $c$ of multiplicity $2$ as well.
Let $u,w$ be the vertices with lists that contain $c$.  So, $L(u)=L(w)=\{a,c\}$. Use $c$ on $u$,
then alternate colors $a,b$ along $P_5$ (skipping $u$) so that $w$ gets $a$.
If $v$ exists, then assign it color $a$ and note that $\eta_L(a)=6$ and $\eta_L(b)=4$.
Thus, whether or not $v$ exists, we have obtained a proportional $L$-coloring of $G$.

In the remaining case, $a$ and $b$ each appear in 3 lists of $P$.
There may be an isolated vertex $v$ with $L(v)=\{a,b\}$ (or not).
If $k=3$, color the endpoints of $P$ by $a$ and the other vertex or
vertices by $b$.
If $k=4$ then there is another color $c$ on two lists of $P$.
We can select $c$ from either list, then since the other list contains $a$ or $b$,
alternate $a,b$ along the rest of $P$.
And if there is a $P_1$ we color it so that $a$ and $b$ are each
used twice overall.
Thus, we have obtained a proportional $L$-coloring when $k$ is~3 or~4.
If $k=5$ then there is either one more color that appears on four lists of $P$
or two more that each appear on two lists of $P$.
We consider these cases separately.

First consider the case where $c$ is a color on four lists of~$P$.
Label the vertices of $P$ in order, $v_1,\ldots,v_5$.
There is a unique vertex of $P$ with list $\{a,b\}$.
If $L(v_2)=\{a,b\}$, color $v_1$ and $v_4$ with $c$,
color $v_2$ and $v_3$ with $a$ and $b$ (one each), color
$v_5$ with either $a$ or $b$ and color the $P_1$ vertex with
the other if it exists. The case $L(v_4)=\{a,b\}$ is similar.
Otherwise, we can color $v_2$ and $v_4$ with $c$, color the other vertices
of $P$ by $a$ and $b$ so that one is used twice and the other once,
and that last color can also be used on the $P_1$ if it exists.
In each case, we get a proportional $L$-coloring.

For the other case, let $c,d$ be the colors that each appear on on two lists of~$P$.
By Lemma~\ref{lem: disjoint}(i) and~(ii), there can be no vertex with list $\{c,d\}$.
Then there is a unique vertex $v$ in $P$ with the list $\{a,b\}$.
There cannot be three consecutive vertices with lists containing $c$ or $d$,
since by Lemma~\ref{lem: disjoint}(i) and~(ii), all
three lists would have to share $a$ or share $b$,
but those colors appear only three times on~$P$, including on~$v$.
Hence $v$ is the middle vertex of $P$,
with lists $\{a,c\},\{a,d\}$ on one side and $\{b,c\},\{b,d\}$ on the other.
Select $c$ or $d$ for one neighbor of $v$; then completing a proportional $L$-coloring of $G$
is straightforward.

A proportional $L$-coloring of $G$ gives us a contradiction, so the proof of Theorem~\ref{thm: full2character}
is complete. \qed

\end{document}